\documentclass{aptpub}

\authornames{S.~Banerjee and W.S.~Kendall} 
\shorttitle{Coupling the Kolmogorov Diffusion} 

 \newcommand{\Title}{Coupling the Kolmogorov Diffusion: \\ maximality and efficiency considerations}

 \newcommand{\Expect}[1]{\operatorname{\mathbb{E}}\left[#1\right]}

 \newcommand{\Ito}{It\^o}
 \newcommand{\Law}[1]{\mathcal{L}\left({#1}\right)}
 \newcommand{\Prob}[1]{\operatorname{\mathbb{P}}\left[#1\right]}
 \newcommand{\Reals}{\mathbb{R}}
 \newcommand{\Var}[1]{\operatorname{Var}\left[#1\right]}

 \renewcommand{\d}{\operatorname{d}}
 \DeclareMathOperator{\dist}{dist}
 \newcommand{\distTV}{\dist_{\textsc{TV}}}

 \newcommand{\avec}{\operatorname{\underline{\mathit{a}}}}
 \newcommand{\ee}{\operatorname{\underline{\mathit{e}}}}
 \newcommand{\origin}{\underline{\text{0}}}
 \newcommand{\xx}{\operatorname{\underline{\mathit{x}}}}
 \newcommand{\yy}{\operatorname{\underline{\mathit{y}}}}
 \newcommand{\zz}{\operatorname{\underline{\mathit{z}}}}
 
 \newcommand{\DDi}{\mathit{D}}
 \newcommand{\DD}{\operatorname{\underline{\underline{\DDi}}}}
 
 \newcommand{\HHi}{\mathit{H}}
 \newcommand{\HH}{\operatorname{\underline{\underline{\HHi}}}}
 
 \newcommand{\EEi}{\mathit{E}}
 \newcommand{\EE}{\operatorname{\underline{\underline{\EEi}}}}
 
 \newcommand{\PPi}{\mathit{P}}
 \newcommand{\PP}{\operatorname{\underline{\underline{\PPi}}}}
 
 \newcommand{\Koli}{\mathit{I}}
 \newcommand{\Kol}{\operatorname{\underline{\Koli}}}
 
 \newcommand{\LLi}{\mathit{L}}
 \newcommand{\LL}{\operatorname{\underline{\underline{\LLi}}}}
 
 \newcommand{\VVi}{\mathit{V}}
 \newcommand{\VV}{\operatorname{\underline{\underline{\VVi}}}}
 
 \newcommand{\WWi}{\mathit{W}}
 \newcommand{\WW}{\operatorname{\underline{\WWi}}}
 
 \newcommand{\ZZi}{\operatorname{\mathit{Z}}}
 \newcommand{\ZZ}{\operatorname{\underline{\ZZi}}}

\begin{document}


\title{\Title} 

\authorone[University of Warwick]{Sayan Banerjee} 

\addressone{Department of Statistics, University of Warwick, Coventry CV4 7AL} 

\authortwo[University of Warwick]{Wilfrid Kendall} 

\addresstwo{Department of Statistics, University of Warwick, Coventry CV4 7AL} 

\setcounter{footnote}{0}
\renewcommand{\thefootnote}{\arabic{footnote}}

\begin{abstract}
This is a case study concerning the rate at which probabilistic coupling occurs for nilpotent diffusions.
We focus on the simplest case of Kolmogorov diffusion 
(Brownian motion together with its time integral or, more generally, together with a finite number of iterated time integrals).
We show that in this case there can be no Markovian maximal coupling. 
Indeed, there can be no \emph{efficient} Markovian coupling strategy (efficient for all pairs of distinct starting values),
where the notion of efficiency extends the terminology of
Burdzy \& Kendall (``Efficient Markovian couplings: examples and counterexamples'', \emph{Annals of Applied Probability}, 2000, 10.2, 362--409).
Finally, at least in the classical case of a single time integral,
it is not possible to choose a Markovian coupling that is optimal in the sense of simultaneously minimizing the
probability of failing to couple by time \(t\) for all positive \(t\).
In recompense for all these negative results, we exhibit a simple efficient non-Markovian coupling strategy.
\end{abstract}

 Brownian motion;
 Brownian time integral;
 co-adapted coupling; 
 coupling;
 efficient coupling;
 filtration;
 finite-look-ahead coupling;
 hypoelliptic diffusion;
 immersed coupling;
 Karhunen-Lo\`eve expansion;
 Kolmogorov diffusion;
 Markovian coupling;
 maximal coupling;
 nilpotent diffusion;
 optimal Markovian coupling;
 reflection coupling;
 synchronous coupling


\section{Introduction} \label{sec:intro}
This paper is written in homage and thanks to our friend and colleague Nick Bingham, who has always made it his mission to encourage and to spur on younger 
colleagues.
It is a case-study of probabilistic coupling for a particular simple non-elliptic diffusion, namely the Kolmogorov diffusion \((B, \int B\,{\d}t)\)
(Brownian motion together with its time integral), studied for example in McKean's celebrated stochastic oscillator paper \cite{McKean-1963}.
The work forms part of a long-running programme of study of coupling for nilpotent diffusions,
by which we mean, diffusions with infinitesimal generators of the form
\[
 \chi_0 + \sum_{i=1}^k \chi_i^2
\]
for smooth vectorfields \(\chi_0\), \(\chi_1\), \(\chi_2\), \ldots, \(\chi_k\) such that the Lie algebra generated by the vectorfields is nilpotent
(iterated Lie brackets of the vectorfields vanish at a sufficiently high order of iteration),
so that consequently the diffusion has smooth positive probability transition densities.
(Note that there is a fully worked out structure theory for diffusions for which the vectorfields form a nilpotent Lie algebra: see for example \cite[Section 4.9]{Kunita-1997}.)
As a case-study this work can be 
usefully compared with the study  \cite{Kendall-2013a} of probabilistic coupling for scalar Brownian motion together with local time at the origin.
In this introductory section, we begin by making a careful definition of an iterated generalization of \((B, \int B\,{\d}t)\),
and then introduce some key coupling concepts.

\subsection{The Kolmogorov diffusion} \label{sec:kolmogorov}
The classical two-component Kolmogorov diffusion is obtained by pairing a real Brownian motion \(B\) with its time integral \(\int B\,{\d}t\).
The vectorfields in question are \(\chi_0=x_1\partial/\partial x_2\) and \(\chi_1=\partial/\partial x_1\), and nilpotence follows from the fact that
\([\chi_1,\chi_0]=\chi_2=\partial/\partial x_2\), and noting that \([\chi_1,\chi_2]=[\chi_0,\chi_2]=0\).
The diffusion \((B, \int B\,{\d}t)\) was studied, for example, by McKean \cite{McKean-1963} as a simple example of a stochastic oscillator; see also \cite{AliliWu-2014}.
Distributional properties have been investigated, for example, in \cite{GroeneboomJongbloedWellner-1999,WellnerSmythe-2001,KearneyMajumdar-2005},
while \(L^p\) estimates are studied in \cite{Yan-2004} when Brownian motion is replaced by a continuous local martingale.
It has arisen in statistical studies: see \cite{AueHorvathHuskova-2009} for an application to polynomial regression.
There are potential applications (for example, to model relativistic diffusion of photons \cite{Bailleul-2008}, also to model the motion of a tracer in fluid flow \cite{JansonsMetcalfe-2005b});
however its main interest is as a simple model for a non-elliptic diffusion.
Coupling properties have been studied in \cite{BenArousCranstonKendall-1995} and numerically in \cite{JansonsMetcalfe-2005b}, 
also (when supplemented by further iterated time integrals) in \cite{KendallPrice-2004};
this problem provides the simplest non-trivial example of a diffusion with nilpotent group symmetries which admits a Markovian or immersed coupling.
Most interest to date has focussed on the classical two-component Kolmogorov diffusion. 
Here we follow \cite{KendallPrice-2004} in considering coupling for (in the most part) the case of index \(k\) (\(k\) iterated time integrals),
since the general structure adds clarity to the arguments.

We begin 
with
an explicit definition.
Given \(B\), a standard real-valued Brownian motion (hence begun at \(0\)), 
we define the \emph{standard} generalized Kolmogorov diffusion 
\(\widetilde{\Kol}=(\widetilde{\Koli}_0, \widetilde{\Koli}_1, \ldots, \widetilde{\Koli}_k)^\top\) of index \(k\) by the finite recursion
\begin{align}
 \widetilde{\Koli}_0(t) \quad&=\quad B(t)\,, \nonumber\\
 \widetilde{\Koli}_r(t) \quad&=\quad \int_0^t \widetilde{\Koli}_{r-1}(s) \,{\d}{s} \qquad \text{for } r=1, 2, \ldots, k\,. 
 \label{eqn:standard-kolmogorov}
\end{align}
Thus \(\widetilde{\Koli}_r\) is simply the \(r\)-fold iterated time-integral of Brownian motion. 
We shall refer to the index \(1\) case as the \emph{classical} Kolmogorov diffusion,
written in column vector form as \((\widetilde{\Koli}_0, \widetilde{\Koli}_1)^\top=(B, \int{B\,{\d}{t}})^\top\).

Nilpotence for \(\widetilde{\Kol}\) follows by considering the infinitesimal generator
\[
\chi_0 + \frac12 \chi_1^2
\]
where \(\chi_1=\partial/\partial x_1\) 
(differentiation in the ``Brownian'' direction, temporarily deviating from the indexing convention of \eqref{eqn:standard-kolmogorov} for the sake of convenience of exposition) 
but \(\chi_0=x_1\chi_2+x_2\chi_3+\ldots+x_{k-1}\chi_k\), where \(\chi_r=\partial/\partial x_r\).
An inductive argument based on \([\chi_r,\chi_0]=\chi_{r+1}\) (for \(r=1, \ldots, k-1\)) shows that \(k^\text{th}\)-iterated Lie brackets vanish.

Bearing in mind the form of the law of \(\widetilde{\Kol}(t+s)\) conditional on \(\mathcal{F}_t=\sigma\{B(u):u\leq t\}\), 
we define the index \(k\) Kolmogorov diffusion \(\Kol\), begun at \(\xx=(x_0,x_1,\ldots, x_k)^\top\), using
\begin{equation}\label{eqn:generalized-kolmogorov}
 \Koli_r(t) \quad=\quad \widetilde{\Koli}_r(t) + x_r + t x_{r-1} + \frac{t^2}{2}  x_{r-2} + \ldots + \frac{t^r}{r!}  x_0\,.
\end{equation}
This definition arises from considering the integral curve of \(\chi_0\) determined by the initial values \(\xx=(x_0,x_1,\ldots, x_k)^\top\) of \(\Kol\) at time \(0\),
and using linearity.

The Kolmogorov diffusion is simple enough to allow for explicit calculations, and yet exhibits interesting properties from the point of view of probabilistic coupling.
Much of the interest of the index \(k\) Kolmogorov diffusion \(\Kol\) lies in the observation that
the random vector \(\Kol(t)\) has a positive continuous density over all \(\Reals^{k+1}\) for any positive time \(t>0\).
This is related of course to H\"ormander's hypoellipticity theorem (since \(\Kol\) is indeed a hypoelliptic diffusion), 
but can be seen much more directly by calculating the covariance structure of \(\Kol(t)\) and noting that the resulting variance-covariance matrix is non-singular.
This computation is carried out in section \ref{sec:calcs}; but first we introduce some relevant concepts from coupling.

\subsection{Maximal couplings} \label{sec:maximal}
The technique of probabilistic coupling dates back to Doeblin \cite{Doeblin-1938}.
However Griffeath \cite{Griffeath-1975} was the first to prove a remarkable optimization result:
it is possible to construct a coupling in a \emph{maximal} way,
in the following sense.
Given two coupled copies \(X\), \(\widetilde{X}\) of a random process, let \(T\) be the (random) \emph{coupling time},
namely \(T=\inf\{t>0: X_s=\widetilde{X}_s \text{ for all } s\geq t\}\). 
Then a given coupling is \emph{maximal} if it simultaneously minimizes \(\Prob{T>t}\) for all \(t>0\);
in particular \(\Prob{T>t}\) is then given by the total variation distance between the distributions of \(X_t\) and \(\widetilde{X}_t\).
Thus a maximal coupling occurs at altogether the fastest possible rate.
Griffeath proved the existence of maximal couplings for discrete-time Markov chains.
Pitman \cite{Pitman-1976} gave an elegant explicit construction for homogeneous discrete-time Markov chains on a countable state-space case 
(and in fact his construction generalizes easily),
while Goldstein \cite{Goldstein-1979} extended the result to general discrete-time random processes. 
For further generalizations see, for example, \cite{SverchkovSmirnov-1990}.

The notion of coupling contains many subtleties. 
For example, in general
(and in contrast to the Markovian case described below)
the simpler random time \(\inf\{t>0: X_t=\widetilde{X}_t\}\) may fail to produce a coupling time; this relates to the notion of \emph{faithful coupling} \cite{Rosenthal-1997}
(note however that constructions of Pitman type deliver maximal couplings for which this simpler random time does produce a coupling).
We note in passing that one can relax the definition of coupling to allow for arbitrary time-shifts: this is the notion of \emph{shift-coupling} \cite{Thorisson-1994}.
In general it is hard to produce explicit constructions of maximal couplings.
On the other hand it is much easier to build \emph{Markovian couplings}: couplings which jointly produce a Markov process, whose transition probability kernel
has as marginals the transition probability kernels of the coupled diffusions.
There is an important detail here: the marginals must be transition probability kernels \emph{with respect to the natural filtration of the coupled process}.
A slightly more general notion of coupling, that of \emph{immersed coupling} (also called \emph{co-adapted coupling}) can be described succinctly and with more clarity:
the martingales for the filtrations of \(X\) and \(\widetilde{X}\) remain martingales for the joint filtration of \((X, \widetilde{X})\) \cite{Kendall-2013a},
so that the filtrations of \(X\) and \(\widetilde{X}\) are immersed in the joint filtration of \(X\) and \(\widetilde{X}\).
However in the following we will restrict ourselves to consideration of Markovian couplings.

Note that the above constructions beg the question of whether the coupling time \(T\) can be chosen to be almost-surely finite, in which case the coupling is said to be \emph{successful}.
In the case of the Kolmogorov diffusion it can be shown that successful Markovian couplings exist \cite{BenArousCranstonKendall-1995,KendallPrice-2004};
The extent to which this is the case for a general nilpotent diffusion is an interesting open question 
(but see \cite{BenArousCranstonKendall-1995,Kendall-2007,Kendall-2009d,Kendall-2013a}).

Markovian couplings are relatively easy to construct and verify (consider for example the classic reflection coupling of Brownian motion);
however in general they will not be maximal (the Brownian reflection coupling provides a rare exception). 
In the case of driftless Brownian motions on rather general spaces, Kuwada \cite{Kuwada-2009} showed that maximal Markovian couplings only occur in highly symmetrical cases; 
it is shown in \cite{BanerjeeKendall-2014} that if a smooth elliptic diffusion on a Riemannian manifold admits
a Markovian maximal coupling then the manifold must be a space form, the diffusion must be Brownian motion plus drift, and the drift must arise from a 
continuous one-parameter group of symmetries of the space form (possibly augmented by dilations in the Euclidean case). 
Thus (at least in the elliptic case) Markovian maximal couplings are very rare.


\subsection{Efficient couplings} \label{sec:efficient}
We have asserted that we should not expect maximal couplings to be readily constructable. 
But for practical purposes it will often suffice to obtain a coupling such that the probability of failing to couple by time \(t\) is {comparable}
(asymptotically in \(t\)) to the probability of failing to couple {maximally} by \(t\). 

Efficient couplings were introduced in \cite{BurdzyKendall-2000} for the case of Markov chains in which 
there is rapid convergence to a stationary distribution: 
a coupling is \emph{efficient} if the (presumed exponential) rate of coupling from generic initial states equals the (presumed exponential) rate of convergence to stationarity.
Here we generalize to cases where a stationary distribution need not exist: instead of considering exponential rates of convergence, we consider the rate of coupling compared to the maximum possible rate.
\begin{definition}\label{def:efficient}
Let $\mu_{x,y}$ be a
successful coupling of two Markov processes $X$ and $Y$ with state space $\mathcal{S}$.
Suppose that \(X\) and \(y\) start from distinct points $x,y \in \mathcal{S}$ respectively, with coupling time $\tau$, and let $d_t(x,y)$ denote the corresponding total variation distance between their distributions. 
We call $\mu_{x,y}$ an \textit{efficient coupling} if there exists a positive constant \(C(x,y)\) 
such that
 \begin{equation}\label{equation:efficient}
 \frac{\mu_{x,y}\{\tau>t\}}{d_t(x,y)} \quad\leq\quad  C(x,y) 
 \end{equation}
for all $t>0$ (note that \(d_t(x,y) \leq \mu_{x,y}\{\tau>t\}\) follows from Aldous' inequality).

We call the family of couplings $\{\mu_{x,y}: x,y \in \mathcal{S}\}$ an \textit{efficient coupling strategy} if $\mu_{x,y}$ is an efficient coupling for every pair of distinct starting points $x,y \in \mathcal{S}$.
\end{definition}
Note that the constant \(C(x,y)\) 
and the maximal coupling rate may, and often do, depend on initial conditions. 
Note also that it is entirely possible for a diffusion to exhibit efficient couplings from some distinct pairs of starting points, 
while still failing to possess an efficient coupling strategy from all possible distinct pairs of starting points.
(Examples can be constructed using diffusions which are ordered pairs of independent component diffusions, of which the first coordinate is efficient, and the second inefficient.)

Even when maximal or efficient Markovian couplings do not exist, it is still possible that there may be a Markovian coupling which is \emph{optimal},
in the sense of simultaneously minimizing \(\Prob{T>t}\) for all $t>0$ over the class of all Markovian couplings. 
Such a coupling exists in the case of simultaneous coupling of Brownian motion and its local time at zero \cite{Kendall-2013a}.

\subsection{Questions} \label{sec:questions}
The purpose of this paper is to address the following questions:
\begin{itemize}
 \item[Q1:] Is there a maximal Markovian coupling for the generalized Kolmogorov diffusion \(\Kol\)?
 The work of \cite{BanerjeeKendall-2014} is primarily concerned with smooth elliptic diffusions and its main results do not directly apply.)
 \item[Q2:] If not, is there an efficient Markovian coupling for the generalized Kolmogorov diffusion \(\Kol\)?
 \item[Q3:] Maximal couplings are typically hard to compute. 
 Is there an efficient non-Markovian coupling for the generalized Kolmogorov diffusion  \(\Kol\) which is relatively easy to describe?
\end{itemize}
We shall also address the question of whether there might be an optimal Markovian coupling, but only in the case of the classic Kolmogorov diffusion \((B, \int B\,{\d}t)\). 
We restrict here to the classic case;
indeed in the case of generalized Kolmogorov diffusions \(\Kol\) of index exceeding \(2\) we know only of implicit
and indirect constructions of Markovian couplings \cite{KendallPrice-2004}, 
while construction of a successful Markovian coupling at index \(2\) is direct but requires a somewhat involved analysis.

Section \ref{sec:intro} describes basic coupling concepts and gives an exact definition of the generalized Kolmogorov diffusion. 
Section \ref{sec:calcs} carries out some basic calculations for the generalized Kolmogorov diffusion,
from which is derived the straightforward Theorem \ref{thm:max-coupling}, which establishes the rate of maximal coupling.
Section \ref{sec:markov} establishes a sequence of negative results: for the generalized Kolmogorov diffusion it is not possible to construct a Markovian maximal coupling
(Theorem \ref{thm:noMarkKol}),
nor are there any efficient Markovian coupling strategies
(Theorem \ref{thm:notefficient}).
Moreover, for the classical Kolmogorov diffusion, Theorem
\ref{thm:noopt} 
(by way of the analytical Theorem \ref{thm:Kolrate} giving coupling rates for the coupling described in \cite{BenArousCranstonKendall-1995})
shows there can be no optimal Markovian strategy (we restrict here to the classical case of \((B, \int B\,{\d}t)\) to avoid considerable potential complexity).
Section \ref{sec:look-ahead} gives a more positive result in the form of Theorem \ref{thm:finite-look-ahead};
working in the class of non-Markovian couplings which look ahead only over bounded intervals of time,
this theorem exhibits a simple but efficient non-Markovian coupling strategy.
Finally, Section \ref{sec:conc} discusses possible future research directions.

\section{Explicit calculations for the Kolmogorov diffusion} \label{sec:calcs}
The Kolmogorov diffusion is a linear Gaussian diffusion, and therefore permits explicit calculation.
From \eqref{eqn:generalized-kolmogorov} the Kolmogorov diffusion \(\Kol(t)\) of index \(k\), using \(\xx\) as initial configuration,  
can usefully be expressed in terms of the standard Kolmogorov diffusion
in vector form:
\begin{equation}\label{eqn:vector}
 \Kol(t) \quad=\quad \widetilde{\Kol}(t) + \HH(t) \xx\,.
\end{equation}
Here the \((k+1)\times(k+1)\) lower-triangular matrix \(\HH(t)\) can be written as \(\HH(t)=\DD(t) \HH \DD(t^{-1})\), 
where \(\DD(t)\) is the \((k+1)\times(k+1)\) diagonal matrix with entries \(1\), \(t\), \ldots, \(t^k\) running down the diagonal, and
\begin{equation}\label{eqn:def-H}
 \HHi_{a,b} \quad=\quad 
 \begin{cases}
 \frac{1}{(a-b)!} & \text{ if } a\geq b\,,\\
  0 & \text{otherwise.}
 \end{cases}
\end{equation}
Note that \(\HH(t+s)=\HH(t)\HH(s)\). Note also that \(\det{\HH}=1\), so that \(\HH\) and \(\HH(t)\) (for \(t>0\)) are non-singular matrices. 

Consider the Kolmogorov diffusion \(\widetilde{\Kol}\) of index \(k\) and begun at \(\origin\).
Mathematical induction establishes the following linear relationship of Volterra integral type:
\[
 \widetilde{\Koli}_k(t) \quad=\quad \int_0^t \frac{(t-s)^{k-1}}{(k-1)!}B(s) \,{\d}{s} \qquad \text{ for }k>0\,.
\]
Fixing \(T\), and defining \(F(T-t)=(T-t)^k/k!\) for \(k>0\), an application of \Ito's formula to \(F(T-t)B(t)\) for standard Brownian motion \(B\) shows that
\[
 \widetilde{\Koli}_k(T) \quad=\quad \int_0^T \frac{(T-t)^{k}}{k!}\d B(t)\,.
\]
In fact this holds for all \(k\geq0\).
Applying \(L^2\) isometry for \Ito{} integrals, we find that for all \(a,b\geq0\)
\[
 \Expect{\widetilde{\Koli}_a(T) \widetilde{\Koli}_b(T)}\quad=\quad \int_0^T \frac{(T-t)^{a}}{a!}\frac{(T-t)^{b}}{b!}\,{\d}{t}\quad=\quad \binom{a+b}{a}\frac{T^{a+b+1}}{(a+b+1)!}\,.
\]
Thus the variance-covariance matrix \(\VV(T)\) for \(\widetilde{\Kol}(T)\), equivalently \({\Kol}(T)\), is given by
\(\VV(T)=T \DD(T) \VV \DD(T)\) where
\begin{equation}\label{eq:var-covar}
 \VVi_{a,b} \quad=\quad \binom{a+b}{a}\frac{1}{(a+b+1)!}\qquad\text{ for }a,b\geq0\,.
\end{equation}
Note that \(\VV\) is non-singular: given a vector \(\avec\) of coefficients, the \(L^2\)-isometry implies
\begin{multline*}
\avec^\top\,\VV\,\avec\quad=\quad\Var{\int_0^1 \sum_{k=0}^n a_k \frac{(1-t)^{k}}{k!}\d B(t)} \\
\quad=\quad
\int_0^1 \left(\sum_{k=0}^n a_k \frac{(1-t)^{k}}{k!}\right)^2 \,{\d}{t}
\quad=\quad
\int_0^1 \left(\sum_{k=0}^n a_k \frac{t^{k}}{k!}\right)^2 \,{\d}{t}\,,
\end{multline*}
and this integral is zero only if  the polynomial \(\sum_{k=0}^n a_k \tfrac{t^{k}}{k!}\) vanishes identically, forcing \(\avec=\underline{0}\).
Thus \(\VV\) is symmetric positive-definite.

Consider the Cholesky decomposition of the symmetric matrix \(\VV\) (unique up to \(\pm\) sign, since \(\VV\) is positive-definite).
This provides a lower-triangular non-singular matrix \(\LL\)
such that
\begin{equation}\label{eq:LQ-sqrt}
 \VV \quad=\quad \LL \LL^\top\,.
\end{equation}
Applying \eqref{eq:var-covar}, \eqref{eq:LQ-sqrt}, and the lower-triangular nature of \(\LL\),
note for future use that the top row of \(\LL\) can be taken to be \((1,0,\ldots,0)\).
There follows a representation of the distribution of \(\Kol(T)\) in terms of a vector \(\WW\) filled with \(k+1\) independent standard Brownian motions: for fixed \(T\) we obtain
\begin{equation}\label{eq:representation}
 \Kol(T) \quad\stackrel{\mathcal{D}}{=}\quad \DD(T)\left( \HH \DD(T^{-1}) \xx + \LL \WW(T)\right)\,.
\end{equation}
(Note however that this equality in distribution cannot be translated into a sample-path equality as \(T\) varies.)

Suppose \(\Kol^{(1)}\), \(\Kol^{(2)}\) are begun at \(\xx^{(1)}\), \(\xx^{(2)}\) respectively. 
Then \eqref{eq:representation} can be used to compute the total variation distance between the Gaussian distributions \(\Law{\Kol^{(1)}(T)}\) and \(\Law{\Kol^{(2)}(T)}\).
These Gaussian distributions have the same variance-covariance matrix, and therefore the total variation distance is given by the following expression,
with \(\zz=\xx^{(1)}-\xx^{(2)}\):
\begin{multline}\label{eq:total-variation}
 \distTV\left(\Law{\Kol^{(1)}(T)}, \Law{\Kol^{(2)}(T)}\right)\;=\;
 \Prob{
 |N(0,1)| \;\le\;  \frac{1}{2\sqrt{T}} \|\LL^{-1}\HH\DD(T^{-1})\zz\|
 }\\
 \quad\leq\quad
 \frac{\|\LL^{-1}\HH\DD(T^{-1})\zz\|}{\sqrt{2\pi T}}
 \,.
\end{multline}
For large \(T\), the bound and the total variation distance are asymptotically equivalent.


The relationship between coupling, maximal coupling and total variation distance \cite{Griffeath-1975,Pitman-1976,Goldstein-1979} 
(in particular Aldous' inequality)
immediately establishes
sharp bounds on the coupling rate for Kolmogorov diffusions, with coupling rate depending on the extent of agreement between low-index initial conditions.
\begin{theorem}[Maximal coupling rate for the generalized Kolmogorov diffusion]\label{thm:max-coupling}
If \(\Kol^{(1)}\), \(\Kol^{(2)}\) are coupled copies of the generalized Kolmogorov diffusion begun at \(\xx^{(1)}\), \(\xx^{(2)}\) respectively, with \(z_0=\ldots=z_{r-1}=0\) and \(z_r\neq0\)
for \(\zz=\xx^{(1)}-\xx^{(2)}\), and $\tau$ denotes the coupling time, then 
\[
 \Prob{\tau>T} \;\geq\; 
 \Prob{
 |N(0,1)| \;\le\;  \frac{1}{2\sqrt{T}} \|\LL^{-1} \HH \DD(T^{-1})\zz\| 
 }
 \;\sim\;
 O\left(\frac{1}{T^{r+\tfrac12}}\right),
\]
and this sharp lower bound is achieved by a maximal coupling between \(\Kol^{(1)}\) and \(\Kol^{(2)}\).
\end{theorem}
\begin{proof}
It is classical that the maximal coupling achieves the upper bound provided by total variation distance \cite{Griffeath-1975,Pitman-1976,Goldstein-1979}.
Hence this result follows directly from \eqref{eq:total-variation} and the lower-triangular nature of \(\HH\) and \(\LL\) and hence of \(\LL^{-1} \HH\).
 From \eqref{eq:total-variation}, \(\Prob{\tau>T}\) is controlled by
 \(\Prob{ |N(0,1)| \;\le\;  \tfrac{1}{2\sqrt{T}}  \|\LL^{-1} \HH \DD(T^{-1})\zz\|} \). 
 But if the first \(r\) indices of \(\zz\) vanish then \(\|\LL^{-1} \HH \DD(T^{-1})\zz\|=O(T^{-r})\),
 and the result follows from the na\"ive bound
 \[
 \sqrt{\frac{2}{\pi}} \; \ell\;e^{-\ell^2/2}
 \quad\leq\quad
 \frac{1}{\sqrt{2\pi}}  \int_{-\ell}^\ell \exp\left(-\frac{u^2}{2}\right)\d{u} 
 \quad\leq\quad 
 \sqrt{\frac{2}{\pi}} \; \ell \,.
 \]
\end{proof}

This has implications for efficient (albeit possibly non-Markovian) coupling strategies for the generalized Kolmogorov diffusion: 
coupling will occur much faster if the initial states give the same initial Brownian locations and agree up to the first \(r-1\geq0\) iterated integrals.


\section{Markovian couplings of Kolmogorov diffusions} \label{sec:markov}
In this section we consider the rate of coupling for \emph{Markovian} couplings of Kolmogorov diffusions:
can Markovian couplings be maximal in this case? or efficient? or, failing either of these, 
can there be Markovian maximal couplings of Kolmogorov diffusions which are optimal in the sense of coupling faster than any other Markovian couplings?
We shall see that the answers to the first two of these questions are negative for Kolmogorov diffusions of whatever positive index.
For the third question we shall consider only classic (index one) Kolmogorov diffusions, and show that 
at least in this case the answer is again negative.

\subsection{Markovian Kolmogorov couplings are never maximal}
For a very general class of Markov processes, if a Markovian maximal coupling exists for two copies of such a process started from distinct points, then this coupling must satisfy some specific properties. 
Maximal coupling has to occur at the space-time interface defined by equality of the transition probabilities of the two coupled diffusions.
If in addition the coupling is Markovian, then Varadhan asymptotics show that at a given time the interface must be of a specific form (a hyperplane, if the diffusive part of the diffusion is Brownian).
Both properties follow from `soft' arguments and do not depend explicitly on the specific process being considered. 
This is the content of Section 1.1 of \cite{BanerjeeKendall-2014} and the results there can be used to show the following:
\begin{theorem}\label{thm:noMarkKol}
There is no Markovian maximal coupling for the Kolmogorov diffusion of any positive index, started from any pair of distinct points.
\end{theorem}
\begin{proof}
Consider the implications of the existence of a Markovian maximal coupling between two index \(k\) generalized Kolmogorov diffusions
\(\Kol^+\) and \(\Kol^-\),
begun at different starting points, \(\xx_0^+\) and \(\xx_0^-\) respectively.
Suppose that the coupling time is \(\tau\).

First consider the case where the last coordinates of \(\xx_0^+\) and \(\xx_0^-\) differ.
By linearity, we may suppose that \(\xx_0^++\xx_0^-=\origin\) (the zero vector).
By \eqref{eq:representation}, for any fixed \(t>0\),
 \[
  \Kol^\pm(t) 
  \quad\stackrel{\mathcal{D}}{=}\quad \DD(t)\left(\LL \WW^\pm(t) + \HH \DD(t^{-1}) \xx_0^{\pm} \right)\,.
 \]
Here \(\WW^\pm\) are \((k+1)\)-dimensional Brownian motions which are coupled (not necessarily in a Markovian manner). 
The Gaussian distributions of \(\Kol^\pm(t)\) have the same non-singular variance matrix \(t\VV(t)=t\DD(t)\VV\DD(t)\):
the corresponding probability densities agree on a hyperplane \(\mathcal{H}(t;\xx_0^+,\xx_0^-)\) which runs through \(\origin\)
and is orthogonal to the vector given by the vector expression
\[
 (t\DD(t)\VV\DD(t))^{-1} \DD(t)\HH\DD(t^{-1})\zz
 \quad=\quad
 t^{-1}\DD(t^{-1}) \VV^{-1}\HH \DD(t^{-1})\zz\,,
\]
where \(\zz=\xx_0^+-\xx_0^-\).
Note that invertibility of \(\VV\) and \(\HH\) imply that this vector is non-zero.

It is convenient to scale this vector expression by \(t^{1+2k}\), and so to deduce that the hyperplane \(\mathcal{H}(t;\xx_0^+,\xx_0^-)\) is normal to the vector
\begin{equation}\label{equation:normal}
 (t^k{\DD}(t^{-1})) \; \VV^{-1}\HH \; (t^k{\DD}(t^{-1}))\zz\,,
\end{equation}
where \(t^k{\DD}(t^{-1})\) is a diagonal \((k+1)\times(k+1)\) matrix whose diagonal is composed of \(t^k\), \(t^{k-1}\), \ldots, \(t\), \(1\).

The hyperplane \(\mathcal{H}(t;\xx_0^+,\xx_0^-)\) separates $\mathbb{R}^{k+1}$ into disjoint half-spaces, one containing $\Kol^+(0)$ and the other containing $\Kol^-(0)$. We call the respective half spaces \(\mathcal{H}^+(t;\xx_0^+,\xx_0^-)\) and \(\mathcal{H}^-(t;\xx_0^+,\xx_0^-)\). Now the following observations must necessarily hold for a Markovian maximal coupling $\mu$ of \(\Kol^\pm\) with coupling time $\tau$:
\begin{enumerate}
\item Denote the probability densities corresponding to $\Kol^\pm(t)$ by $p^\pm_t(\xx_0^\pm,\cdot)$. Let $\alpha_t(\cdot)=p^+_t(\xx_0^+,\cdot)-p^-_t(\xx_0^-,\cdot)$.
Then \cite[Lemma 2]{BanerjeeKendall-2014} states that, for any Borel measurable set $A \in \mathbb{R}^{k+1}$,
\begin{eqnarray*}
\mu\{\Kol^\pm(t) \in A, \tau > t\} \quad=\quad \int_A\alpha^\pm_t(\zz) \ \d \zz
\end{eqnarray*}
where $\alpha^+(\zz)=\max\{p^+_t(\xx_0^+,\zz)-p^-_t(\xx_0^-,\zz),0\}$ and $\alpha^-(\zz)=\max\{p^-_t(\xx_0^-,\zz)-p^+_t(\xx_0^+,\zz),0\}$.
\linebreak
Thus, if coupling has not been successful by time \(t\), then at time \(t\) the 
diffusions \(\Kol^\pm(t)\) must lie on different sides of the hyperplane \(\mathcal{H}(t;\xx_0^+,\xx_0^-)\).
Moreover we may use the Gaussian nature of \(\Kol^\pm(t)\) to deduce 
that under the conditioning \(\tau>t\) the support of $\Kol^+(t)$ (respectively $\Kol^-(t)$) must be the whole of \(\mathcal{H}^+(t;\xx_0^+,\xx_0^-)\) 
(respectively \(\mathcal{H}^-(t;\xx_0^+,\xx_0^-)\)).
\item Let \(\mu_t\) denote the joint law of the coupled \(\Kol^\pm\) evaluated at time \(t\).
The coupling is Markovian and maximal, so \cite[Lemma 3]{BanerjeeKendall-2014} shows that
for almost every pair \((\xx^+,\xx^-)\) of distinct points in the support of \(\mu_t\) it must be the case that the forward processes \(\Kol^\pm(t+\cdot)\)
must generate a new Markovian maximal coupling starting from \(\xx^+\) and \(\xx^-\) respectively. We will denote the set of such \((\xx^+,\xx^-)\) by \(\mathcal{M}(\mu_t)\).
\end{enumerate}

The first observation shows that, for any positive \(t\), \(s\), when conditioned 
on \(\tau>t+s\),
the support of the conditional distribution of \(\Kol^+(t+s)\) must be the whole of 
 \(\mathcal{H}^+(t+s;\xx_0^+,\xx_0^-)\).
Adding the second observation, we may deduce that for all \((\xx^+,\xx^-)\in\mathcal{M}(\mu_t)\), when conditioning on
\(\Kol^+(t)=\xx^+\) and \(\Kol^-(t)=\xx^-\), then the support of the conditional distribution of \(\Kol^+(t+s)\) (also given \(\tau>t+s\)) must be the whole of  \(\mathcal{H}^+(s;\xx^+,\xx^-)\),
where \(\mathcal{H}(s;\xx^+,\xx^-)\) is the hyperplane on which the densities at time \(s\) agree for two generalized Kolmogorov diffusions begun
at \(\xx^+\) and \(\xx^-\) respectively.
Thus for \((\xx^+,\xx^-)\in\mathcal{M}(\mu_t)\) we must have 
 \(\mathcal{H}^+(t+s;\xx_0^+,\xx_0^-)=\mathcal{H}^+(s;\xx^+,\xx^-)\).

Arguing as above, the common hyperplane must be normal to the vector
\begin{multline*}
 (s^k{\DD}(s^{-1}))\; \VV^{-1}\HH\; (s^k{\DD}(s^{-1}))(\xx^+-\xx^-)
\\
\quad=\quad
 (s^k{\DD}(s^{-1}))\; \VV^{-1} \; (s^k{\DD}(s^{-1})) \times (s^{-k}{\DD}(s))\;\HH\; (s^k{\DD}(s^{-1}))(\xx^+-\xx^-)
 \,.
\end{multline*}
Now consider the limit of this vector as \(s\downarrow 0\).
We see that \((s^k{\DD}(s^{-1}))\; \VV^{-1} \; (s^k{\DD}(s^{-1}))\) converges to a matrix all of whose entries are zero save for 
the \((k,k)\) position.
Moreover this entry is that same as the \((k,k)\) entry of \(\VV^{-1}\): 
since \(\VV\) is symmetric positive-definite, it follows that \(\VV^{-1}\) is also symmetric positive-definite, and therefore its \((k,k)\) entry must be positive.

On the other hand, consider 
the lower-triangular matrix \((s^{-k}{\DD}(s))\;\HH\; (s^k{\DD}(s^{-1}))\).
All off-diagonal entries 
must tend to zero with \(s\); the on-diagonal entries are not affected and, by \eqref{eqn:def-H}, all are equal to \(1\).

These facts, along with the assumption that the last coordinate of $\zz$ is non-zero, imply
that as \(s\downarrow0\) so \((s^k{\DD}(s^{-1}))\; \VV^{-1}\HH \;(s^k{\DD}(s^{-1}))(\xx^+-\xx^-)\) converges to a vector parallel to \((0, 0, \ldots, 0, 1)^\top\),
and this vector is non-zero if \(s>0\).
Consequently
\(\mathcal{H}(t;\xx_0^+,\xx_0^-)\) is normal to \(\ee_k=(0, 0, \ldots, 0, 1)^\top\) for all positive \(t\).

Together with \eqref{equation:normal}, this tells us that, for each $t>0$, 
there is a non-zero scalar $\lambda_t$ such that 
$$
(t^k{\DD}(t^{-1}))\; \VV^{-1} \; (t^k{\DD}(t^{-1})) \times (t^{-k}{\DD}(t))\;\HH\; (t^k{\DD}(t^{-1}))\zz\quad=\quad\lambda_t \ee_k 
$$ 
or equivalently
\begin{equation}\label{equation:maxeqnone}
(t^{-k}{\DD}(t))\;\HH\; (t^k{\DD}(t^{-1}))\zz\quad=\quad\lambda_t(t^{-k}{\DD}(t))\; \VV \; (t^{-k}{\DD}(t))\ee_k\,.
\end{equation}
For any $0 \le i \le k$, (\ref{equation:maxeqnone}) yields
\begin{equation}\label{equation:maxeqntwo}
\sum_{j=0}^i t^{i-j}H_{i,j}z_j\quad=\quad t^{-k+i}V_{i,k}\lambda_t\,.
\end{equation}
Putting $i=k$ in (\ref{equation:maxeqntwo}), we find
$$
\lambda_t\quad=\quad V_{k,k}^{-1}\sum_{j=0}^k t^{k-j}H_{k,j}z_j\,.
$$ 
Substituting this value of $\lambda_t$ back into (\ref{equation:maxeqntwo}), we obtain the following equation:
$$
\sum_{j=0}^i t^{i-j}H_{i,j}z_j\quad=\quad \frac{V_{i,k}}{V_{k,k}}\sum_{j=0}^k t^{i-j}H_{k,j}z_j\,.
$$
Setting $i=0$ in the above equation and comparing coefficients of inverse powers of $t$, 
the explicit formulae \eqref{eqn:def-H} for $\HH$ and \eqref{eq:var-covar} for $\VV$ lead to $\zz=0$,
and we thus obtain a contradiction of the initial hypothesis that \(z_k\neq0\).

If the last coordinates of \(\xx_0^+\) and \(\xx_0^-\) agree, then consider the largest $j<k$ such that $\Koli^+_j(0) \neq \Koli^-_j(0)$. 
Without loss of generality, suppose $\Koli^+_j(0) > \Koli^-_j(0)$. 
Path continuity of the diffusion shows that there must be $T>0$ such that the measure $\mu\{\Koli^+_j(t) - \Koli^-_j(t)>0 \text{ for all } t\le T\}$ must be positive. 
But
$$
\Koli^+_k(t)-\Koli^-_k(t)\quad=\quad\int_0^{t}\int_0^{s_{k-1}}\cdots\int_0^{s_{j+1}}\left(\Koli^+_j(s_j) - \Koli^-_j(s_j)\right){\d}s_j\dots {\d}s_{k-1}\,,
$$
and therefore $\mu\{\Koli^+_k(t) - \Koli^-_k(t)>0 \text{ for all } 0<t\le T\}>0$.

In particular, $\mu_T\{\Koli^+_k(T) \neq \Koli^-_k(T)\}>0.$ 
Under the hypothesis of existence of a Markovian maximal coupling starting from  \(\xx_0^+\) and \(\xx_0^-\),
and using \cite[Lemma 3]{BanerjeeKendall-2014},
we can find $(\xx^+_T, \xx^-_T) \in \mathcal{M}(\mu_T)$ such that the coupled forward processes
$(\Koli^+_k(T+\cdot), \Koli^-_k(T+\cdot))$ started from $(\xx^+_T, \xx^-_T)$ create a Markovian maximal coupling. 
The previous argument applied to this coupling then leads to a contradiction.
\end{proof}

\subsection{Markovian Kolmogorov couplings are never efficient}
As before, let $d_t(\xx,\yy)$ denote the total variation distance between the laws of generalized Kolmogorov diffusions 
(of the same positive index \(k\)) started from $\xx$ and $\yy$ respectively. 
From Theorem \ref{thm:max-coupling}, if the first coordinates of $\xx$ and $\yy$ disagree then $d_t(\xx,\yy)  \sim t^{-1/2}$,
while if they agree then $d_t(\xx,\yy)  \sim t^{-3/2}$ or even higher powers of \(t^{-1}\) (if further coordinates agree).
That is to say, \textit{for efficient (possibly non-Markovian) coupling strategies (see Definition \ref{def:efficient}), 
the coupling happens much faster when the scalar Brownian motions, driving the coupled Kolmogorov diffusions of index \(k\), both start from the same point}. 
This turns out to be the key observation in proving Theorem \ref{thm:notefficient}. 
The theorem, on the absence of efficient Markovian coupling strategies,
will follow as a direct corollary of the following lemma.
\begin{lem}\label{lem:TVlbound}
Let $\mu$ be any successful Markovian coupling of Kolmogorov diffusions starting from distinct points $\xx$ and $\yy$ whose first coordinates agree. Let $\tau$ denote the coupling time. Then there are positive constants $C_{\mu}, t_{\mu}$ (possibly depending on the coupling $\mu$) such that, for all $t \ge t_{\mu}$,
\begin{equation}\label{equation:noteffeq}
\mu\{\tau>t\} \quad\ge\quad \frac{C_{\mu}}{\sqrt{t}}\,.
\end{equation}
\end{lem}
\begin{proof}
Let \(\Kol^{(1)}\) and \(\Kol^{(2)}\) denote specific Markovian-coupled copies of the generalized Kolmogorov diffusions
starting from $\xx$ and $\yy$ (with first coordinates agreeing but second coordinates disagreeing). 
This corresponds to a Markovian coupling of the driving Brownian motions \(B\) and \(\widetilde{B}\).
Let $\mathcal{F}_t=\sigma\{\Kol^{(1)}(s), \Kol^{(2)}(s)\;:\; s\leq t\}$ be the corresponding \(\sigma\)-algebra of events determined by time \(t\).
Consider the possibility that $\mu\{B(t)=\tilde{B}(t)\}=1$ for all $t$.
A Fubini argument then shows that 
\begin{equation*}
\mu\{B(t)=\tilde{B}(t) \text{ for almost every } t>0\}\quad=\quad1\,.
\end{equation*}
By path continuity of Brownian motion, the above would imply that \(B\) and \(\widetilde{B}\) are synchronously coupled $\mu$-almost surely and hence $\mu\{\Kol^{(1)}(t)-\Kol^{(2)}(t)=\xx-\yy \text{ for all} \ t>0\}=1$. 
So a successful Markovian coupling cannot be obtained in this manner. 
Thus for a successful Markovian coupling \(\mu\) there must exist $t_0>0$ such that $\mu\{B(t_0)\neq \tilde{B}(t_0)\}>0$. Now, for every $t>t_0$, 
$$
\mu\left\{\tau>t\right\} \quad=\quad \mathbb{E}_{\mu}\left[\;\mathbb{E}_{\mu}\left[\mathbb{I}(\tau>t) \mid \mathcal{F}_{t_0}\right]\;\right]\,,
$$ 
where $\mathbb{E}_{\mu}$ represents expectation with respect to the probability measure $\mu$. 
Introduce the ordinary Brownian coupling time $\tau^*=\inf\{s>t_0:B(s)=\tilde{B}(s)\}$. 
Evidently this happens no later than the Kolmogorov coupling time, so $\mathbb{I}(\tau>t) \ge \mathbb{I}(\tau^*>t)$. 
As the coupling is Markovian, 
the shifted process $\left((\theta_{t_0}B(s),\theta_{t_0}\tilde{B}(s)): s>0\right)$ 
(when conditioned on $\mathcal{F}_{t_0}$)
gives a coupling of Brownian motions starting from $(B(t_0), \tilde{B}(t_0))$. 
If $d^{\textsc{Br}}_t(x,y)$ represents the total variation distance between the distributions of Brownian motions starting from $x$ and $y$ at time $t$, 
then we know that for all $t>0$ and all $x,y$ satisfying $|x-y| \le 2\sqrt{t}$,
\begin{equation*}
d^{\textsc{Br}}_t(x,y) \quad=\quad \Prob{|N(0,1)| \quad\le\quad \frac{|x-y|}{2\sqrt{t}}}
\quad\ge\quad \frac{1}{\sqrt{2 \pi e}} \frac{|x-y|}{\sqrt{t}}\,.
\end{equation*}
There is $t_{\mu} > t_0$ such that, for all $t \ge t_{\mu}$,
$$
\mathbb{E}_{\mu}\left[|B(t_0)-\tilde{B}(t_0)| \;;\; |B(t_0)-\tilde{B}(t_0)|\le 2\sqrt{t-t_0}\right] \quad\ge\quad \sqrt{2 \pi e} C_{\mu}\,,
$$
where $C_{\mu}=\frac{1}{2\sqrt{2\pi e}}\mathbb{E}_{\mu}|B(t_0)-\tilde{B}(t_0)|$. (The left-hand side converges to \(2\sqrt{2\pi e}C_\mu\) as \(t\to\infty\).) 
Thus, for $t \ge t_{\mu}$,
 \begin{align*}
\mathbb{E}_{\mu}\left[\; \mathbb{E}_{\mu}\left[\mathbb{I}(\tau>t) \mid \mathcal{F}_{t_0}\right]\;\right] \quad&\ge\quad 
\mathbb{E}_{\mu}\left[\;\mathbb{E}_{\mu}\left[\mathbb{I}(\tau^*>t) \mid \mathcal{F}_{t_0}\right]\;\right]
 \quad\ge\quad \mathbb{E}_{\mu}\left[d_{t-t_0}^{\textsc{Br}}(B(t_0), \tilde{B}(t_0))\right]\\
 \quad&\ge\quad \frac{1}{\sqrt{2 \pi e}}\ \mathbb{E}_{\mu}\left[\frac{|B(t_0)-\tilde{B}(t_0)|}{\sqrt{t-t_0}} \;;\; \frac{|B(t_0)-\tilde{B}(t_0)|}{\sqrt{t-t_0}} \le 2\right] \\
 \quad&\ge\quad C_{\mu}\,(t-t_0)^{-1/2} 
 \quad\ge\quad C_{\mu}\,t^{-1/2}\,,
 \end{align*}
which proves the lemma. 
\end{proof}
 \begin{theorem}\label{thm:notefficient}
There does not exist any efficient Markovian coupling strategy (in the sense of Definition \ref{def:efficient}) for the Kolmogorov diffusion of index \(k\) (for \(k>0\)).
 \end{theorem}
\begin{proof}
We argue by contradiction.
If there is such an efficient Markovian coupling strategy, then there must exist an efficient Markovian coupling $\mu$  
of the generalized Kolmogorov diffusions
starting from $\xx$ and $\yy$ (with first coordinates agreeing but second coordinates disagreeing) with coupling time $\tau$. Then, by Lemma \ref{lem:TVlbound}, there exist positive constants $C_{\mu}, t_{\mu}$ such that $\mu\{\tau>t\} \ge C_{\mu}t^{-1/2}$ for all $t \ge t_{\mu}$.

But, as noted in the discussion preceding Lemma \ref{lem:TVlbound}, since the first coordinates of \(\xx\) and \(\yy\) agree we have
$$
d_t(\xx,\yy)  \quad\sim\quad t^{-3/2}\,,
$$ 
or even faster.
Thus efficiency of the coupling strategy (in the sense of Definition \ref{def:efficient}) must fail, proving the theorem.
\end{proof}

\begin{remark}
Theorem \ref{thm:notefficient} shows that any Markovian coupling strategy is non-efficient,
in the sense that there exist some pairs of distinct starting points from which it is impossible to construct efficient couplings. 
But it {does not imply} Theorem \ref{thm:noMarkKol}, which shows non-maximality of Markovian couplings from \textit{any pair of distinct starting points}.
\end{remark}

\subsection{Markovian classic Kolmogorov couplings cannot be optimal}
To begin with, recall the simplest version of the Markovian coupling for the classic Kolmogorov diffusion \cite{BenArousCranstonKendall-1995,KendallPrice-2004}.
Write \(U\) for the process corresponding to the difference between the Brownian motions, and \(V\) for the difference in the time integrals of the Brownian motions.
We assume $U_0=1$ and $V_0=0$; generalization to the case of arbitrary distinct starting points should be clear.
We write the coupling probability measure as $\hat{\mu}_{(u,v)}$ when starting values are $U_0=u$, $V_0=v$.

We describe the coupling strategy of \cite{BenArousCranstonKendall-1995,KendallPrice-2004} for \(\hat{\mu}_{(1,0)}\).
When \(U\) and \(V\) have the same sign, we apply reflection coupling (so that \(U\) evolves as a Brownian motion run at rate \(4\)). 
Thus the visits of \((U,V)\) to the axis \(V=0\) have to occur at isolated instants of time: between each pair of visits the particle \((U,V)\) describes a ``half-cycle'' about the origin.
Over the \(k^\text{th}\) cycle, we actually apply reflection coupling until \(U\) hits \(U=\pm2^{-k}\) or \(V=0\), 
taking the sign \(\pm\) as the opposite of the sign of \(U\) at the start of the half-cycle.
We then apply synchronous coupling till \(V=0\), 
so that \(U\) is held constant.
As a result,
\(|U|\) will be no larger than \(2^{-k}\) at the end of the \(k^\text{th}\) half-cycle.

It is convenient to introduce some notation.
Suppose that the \(k^\text{th}\) cycle begins at time \(S_{k-1}\) (so \(S_0=0\)).
Then \(S_k=\inf\{t>S_{k-1}: V(t)=0\}\).
Let
\(T_k=\min\{S_k, \inf\{t>S_{k-1}: U(t)=(-1)^k 2^{-k}\}\}\) be the first time that \(U\) hits \(U=(-1)^k2^{-k}\), or \(S_k\) if that happens first.
Thus \([T_k, S_k)\) is the part of the half-cycle in which synchronous coupling is applied: one might think of this as the ballistic phase,
while \([S_{k-1}, T_k)\) is the Brownian phase. Note that the ballistic phase will be trivial if the Brownian phase hits the \((V=0)\)-axis (that is, \(T_k=S_k\)).

The 
time of successful coupling is the time at which \((U,V)\) hits the origin, and thus it is given by
\[
\tau \quad=\quad \lim_{k \to \infty} S_k \quad=\quad \sum_{k=0}^{\infty}(S_{k+1}-S_k)\,.
\]
Borel-Cantelli arguments show that this limit is finite \cite{BenArousCranstonKendall-1995}. 
Our first task is to determine the precise rate of decay of the probability of failing to couple by time $t$. 


\begin{theorem}\label{thm:Kolrate}
Under the coupling $\hat{\mu}_{(1,0)}$ described above, the coupling time $\tau$ satisfies 
\begin{equation}\label{eq:Kolrate}
\frac{C_1}{t^{1/3}} \quad\le\quad \hat{\mu}_{(1,0)}\left\{\tau>t\right\} \quad\le\quad \frac{C_2}{t^{1/3}} \qquad\text{ for } t\geq1\,,
\end{equation}
for some positive constants $C_1, C_2$.
\end{theorem}
\begin{proof}
Note that scaling arguments show that \(S_{k+1}-S_k\) has the same distribution as \(2^{-k}S_1\), 
so that the principal computation concerns the rate of decay of \(\hat{\mu}_{(1,0)}\{S_1>t\}\). 

Now \(S_1=T_1 + 2V(T_1)=2\int_0^{T_1} (\tfrac{1}{2} + U(s)){\d}s\), since either \(V(T_1)=0\) 
(when \(T_1=S_1\)) or the velocity of \(V\) over \([T_1,S_1]\) is fixed at \(-\tfrac{1}{2}\). 
Writing \(\tfrac{1}{2}+U=2\widetilde{W}\) for a Brownian motion \(\widetilde{W}\) started from $\frac{3}{4}$ (over the time interval \([S_0, T_1]\)),
and noting that \(S_1=T_1+2 V(T_1)=4\int_0^{T_1}\widetilde{W}{\d}t\) (since $T_1=\inf\{t>0: \widetilde{W}_t=0\}$), we see that \(S_1/4\) can be viewed as 
having the density of the random area under a one-dimensional Brownian motion started from a positive level and measured till the first time it hits zero. 

The density for this random area can be obtained from the calculations of \cite[equation (12)]{KearneyMajumdar-2005}.
Consider a standard (rate \(1\)) Brownian motion \(W\) started from $a>0$.
If $\sigma_a= \inf\{t>0: W_t=0\}$,
then
the martingale \(\Expect{\exp(-\lambda\int_0^{\sigma_a} W(s){\d}{s})| \mathcal{F}_t}\) can be used to 
show that the moment generating function
of \(\int_0^{\sigma_a} W(s){\d}{s}\) (viewed as a function of the  starting position \(a\))
must solve
an Airy partial differential equation. The moment generating function can be inverted to yield the distribution of \(\int_0^{\sigma_a}W(s){\d}{s}\):
\begin{equation}\label{eq:KearneyMaj}
\Prob{\int_0^{\sigma_a}W(s){\d}{s}\in {\d}{u}}
\quad=\quad
\frac{2^{1/3}}{3^{2/3}\Gamma\left(\frac{1}{3}\right)}\frac{a}{u^{4/3}}\exp\left(-\frac{2a^3}{9u}\right) {\d}{u}\,.
\end{equation}
From \eqref{eq:KearneyMaj}, it follows that for $t \ge 1$,
\begin{equation}\label{equation:Arhit}
\,\frac{6^{1/3}e^{-\frac{2a^3}{9}}}{\Gamma\left(\frac{1}{3}\right)}\frac{1}{t^{1/3}} \quad\le\quad 
\Prob{\int_0^{\sigma_a}W(s){\d}{s}>t}\quad\le\quad \frac{6^{1/3}}{\Gamma\left(\frac{1}{3}\right)}\frac{1}{t^{1/3}}\,.
\end{equation}
As \(S_1/4\) has the density \eqref{eq:KearneyMaj} with \(a=\tfrac{3}{4}\), \eqref{equation:Arhit} gives us the following for $t \ge 1$:
\begin{equation}\label{equation:Stail}
\,\frac{24^{1/3}e^{-\frac{3}{32}}}{\Gamma\left(\frac{1}{3}\right)}\frac{1}{t^{1/3}} \quad\le\quad 
 \hat{\mu}_{(1,0)}\{S_1>t\} \quad\le\quad \frac{24^{1/3}}{\Gamma\left(\frac{1}{3}\right)}\frac{1}{t^{1/3}}\,.
\end{equation}
We now apply this rate computation to \(\tau=\sum_{k=0}^\infty(S_{k+1}-S_k)\).
By Boole's inequality, scaling, and \eqref{equation:Stail}, if \(t\geq1\) then
\begin{multline*}
 \hat{\mu}_{(1,0)}\left\{\sum_{k=0}^\infty(S_{k+1}-S_k) > t\right\}
 \quad\leq\quad 
 \sum_{k=0}^\infty \hat{\mu}_{(1,0)}\left\{S_{k+1}-S_k > \frac{\sqrt{2}-1}{\sqrt{2}}2^{-k/2}\;t\right\}\\
  \quad=\quad 
 \sum_{k=0}^\infty \hat{\mu}_{(1,0)}\left\{S_1 > \frac{\sqrt{2}-1}{\sqrt{2}}2^{k/2}\;t\right\}
 \quad\leq\quad
  \frac{24^{1/3}}{\Gamma\left(\frac{1}{3}\right)} \left(\frac{\sqrt{2}}{\sqrt{2}-1}\right)^{1/3}\frac{1}{1-2^{-1/6}}\;\frac{1}{t^{1/3}}\,,
\end{multline*}
and this establishes the right-hand inequality in \eqref{eq:Kolrate}. 
On the other hand \eqref{equation:Stail} shows that if \(t\geq1\) then
\begin{multline*}
 \hat{\mu}_{(1,0)}\left\{\sum_{k=0}^\infty(S_{k+1}-S_k) > t\right\}
 \quad\geq\quad 
 \hat{\mu}_{(1,0)}\left\{S_1 > t\right\}
 \quad\geq\quad
 \frac{24^{1/3}e^{-\frac{3}{32}}}{\Gamma\left(\frac{1}{3}\right)}\frac{1}{t^{1/3}} 
 \,,
\end{multline*}
and this establishes the left-hand inequality in \eqref{eq:Kolrate}, thus completing the proof. 
\end{proof}

The order of decay of the failure probability of coupling by time $t$ is $t^{-1/3}$, therefore
this Markovian coupling is far from efficient even when the Brownian motions start from different points. 
In principle similar analyses should be possible for Markovian couplings of generalized Kolmogorov diffusions, though in such cases
progress is difficult because we only know of implicitly defined Markovian coupling constructions for index exceeding \(2\),
while analysis for the case of index \(2\) is complicated \cite{KendallPrice-2004}.

We now consider the interesting question, can there be an \emph{optimal coupling rate}, optimizing over all Markovian couplings
but obtained by a single Markovian coupling of a Kolmogorov diffusion?
For a study of this question in the different context of coupling for Brownian motion together with a local time, see \cite{Kendall-2013a}. 
In contrast to the case of \cite{Kendall-2013a}, 
we will prove that
such a coupling cannot exist for the classical Kolmogorov diffusion.

\begin{theorem}\label{thm:noopt}
There does not exist an optimal Markovian coupling strategy for two copies of the classical Kolmogorov diffusion.
\end{theorem}
\begin{proof}
Fix attention on classical Kolmogorov diffusions which initially differ only in their second coordinate.
Denote by $U$ the difference between the Brownian motions and by $V$ the difference in the time integrals. 
By scaling and stopping arguments, we may concentrate on the situation in which \(U_0=0\) and \(V_0=1\).
Our strategy will be to obtain for each $t>0$, a specific Markovian coupling $\mu^{(t)}$ with coupling time $\tau^{(t)}$ such that $\mu^{(t)}\{\tau^{(t)}>t\} \le \frac{C}{t}$ for all $t>0$, where the constant $C$ does not depend on $t$. 
If there were to exist an optimal Markovian coupling $\nu$ with coupling time $\tau^*$, 
then the above would imply that it should satisfy $\nu\{\tau^*>t\} \le \frac{C}{t}$ for all $t >0$.
This would contradict Lemma \ref{lem:TVlbound}, 
which shows that for any (successful) Markovian coupling $\nu$ there must exist positive constants $C_{\nu}, t_{\nu}$ such that for all $t \ge t_{\nu}$, $\nu\{\tau^*>t\} \ge C_{\nu}t^{-1/2}$.

First we must describe the coupling $\mu^{(t)}$. 
The coupling $\mu^{(t)}$ differs from the coupling described at the start of this subsection essentially only
by early completion of the Brownian phase of its initial half-cycle:
\begin{enumerate}
\item Couple the Brownian motions by reflection till time
\[
T'_1\quad=\quad\inf\{t>0: U_t = -\tfrac{4}{t}\} \wedge \inf\{t>0: V_t =0\}\,. 
\]
Then couple synchronously till time 
\[
S_1\quad=\quad\inf\{t\ge T'_1: V_t=0\}\,. 
\]
\item Starting from $(U_{S_1},0)$, employ the coupling $\hat{\mu}_{(U_{S_1},0)}$ described at the start of this subsection.
\end{enumerate}
In the following, $C_1, C_2, \dots$ will be positive constants that do not depend on $t$. 

First observe that $T'_1$ is smaller than the hitting time of the (rate 4) Brownian motion \(U\) on the level $\frac{4}{t}$ when started at \(0\).
Therefore arguments using the reflection principle show that $\mu^{(t)}\{T'_1>1\} \le \frac{C_1}{t}$.
On the other hand,
\[
V(T'_1)\quad=\quad\int_0^{T'_1} (U(s) + 4/t){\d}s -(4/t)T'_1\quad\leq\quad \int_0^{T'_1} (U(s) + 4/t){\d}s\,,
\]
and we can apply
\eqref{eq:KearneyMaj}
to deduce that
$\mu^{(t)}\{V_{T'_1}>2\} \le \frac{C_2}{t}$. 
Now $S_1-T'_1 \le \frac{t}{2}$ on the event $\{V_{T'_1} \le 2\}$. 
By scaling arguments and Theorem \ref{thm:Kolrate},
\begin{eqnarray*}
\mu^{(t)}\left\{\tau^{(t)}-S_1>\frac{t}{2}\right\} \quad&\le\quad& 
\hat{\mu}_{(\frac{4}{t},0)}\left\{\tau>\frac{t}{2}\right\}\\
\quad&=\quad&\hat{\mu}_{(1,0)}\left\{\frac{4^2}{t^2}\tau>\frac{t}{2}\right\} 
\quad\le\quad \frac{C_3}{t},
\end{eqnarray*}
where $\tau$ denotes the coupling time under the coupling strategy 
described at the start of this subsection.

Combining the above facts, we obtain $\mu^{(t)}\{\tau^{(t)}>t+1\} \le \frac{C_4}{t}$. 
The theorem follows.
\end{proof}

\section{Efficiency for the finite-look-ahead coupling} \label{sec:look-ahead}

Despite these negative results, nevertheless it is possible to exhibit a simple and explicit efficient coupling strategy if we allow couplings which are allowed finite but varying amounts of precognition. 
In this section, we will describe a simple non-Markovian coupling strategy which achieves efficiency.
Our approach will be to divide time into successive intervals $[S_n, S_{n+1}]$ of growing size and then to couple the driving Brownian motions 
(and hence the Kolmogorov diffusions) according to a non-Markovian recipe on each such interval. 
We call this coupling a \emph{finite-look-ahead coupling} as the coupling construction on each interval $[S_n,S_{n+1}]$, although non-Markovian, requires information on the driving Brownian paths only till time $S_{n+1}$. 
Further, the coupling of the future paths of the Kolmogorov diffusions conditional on the paths run up till time $S_n$ depends on the past only through the values taken at time $S_n$. 
Thus the coupling restricted to times $S_n$ (for ${n=0,1,\ldots}$) can be considered to have a Markovian property.  

Recall that in \eqref{eq:representation}, we wrote down a representation for the Kolmogorov diffusion in terms of a Brownian motion vector $\WW$ at time $T$. Note in particular that the lower-triangular form of \(\LL\) means that \eqref{eq:representation} represents \(\Koli_0(T)\) by \(\Koli_0(T)\stackrel{\mathcal{D}}{=}x_0+\WWi_0(T)\).
However this holds only for the stipulated fixed time \(T\); 
we cannot maintain the representation \eqref{eq:representation}
of \(\Kol(T)\) in terms of the full vector \(\WW\) of independent standard Brownian motions 
while simultaneously writing \(\Koli_0(t)=x_0+\WWi_0(t)\) for \(0\leq t \leq T\). 
The representation \eqref{eq:representation} can be best understood as a fragment of an infinite-dimensional representation as follows. 
Consider the initial Brownian path \(\{B(t):0\leq t\leq T\}\) as a Gaussian vector in the infinite dimensional space \(C([0,T])\) of
continuous paths. 
Realize this as the evaluation at time \(T\) of an infinite-dimensional Brownian motion \(\mathcal{B}\) starting at the zero path and taking values in the Banach space \(C([0,T])\),
and evolving in ``algorithmic time''
(as opposed to the ``process time'' \(t\) used for the stochastic process \(\{B(t):0\leq t\leq T\}\)).
A candidate for this is given by
the Karhunen-Lo\`eve expansion
\begin{equation}\label{equation:infBM}
\mathcal{B}(\zeta,t)\quad=\quad
x_0 + \sum_{k=1}^{\infty}\sqrt{\lambda_k}\, w_k(\zeta)\, f_k(t/T)
\end{equation}
for $\zeta, t \in [0,T]$,
where $\lambda_k=\frac{1}{\left(k-\frac{1}{2}\right)^2\pi^2}$, $f_k(t)=\sqrt{2}\sin\left(\left(k-\frac{1}{2}\right)\pi t\right)$ 
denote the eigenvalues and eigenfunctions respectively of the covariance kernel of Brownian motion viewed as a bounded operator acting on $L^2[0,1]$, 
and the $w_k$'s are independent standard Brownian motions. 
Here $\zeta$ represents the algorithmic time and $t$ represents the process time.  
(See \cite{DaPratoZabczyk-2008} for more on stochastic analysis on infinite dimensional spaces.)

Set $\Koli_0(t)=x_0+ \mathcal{B}(T,t)$. 
For $1 \le r \le k$ and $t \in [0,T]$, the iterated integrals \(\Koli_r(t)\) can be viewed as $\mathcal{I}_r(T,t)$ 
where $\{\mathcal{I}_r(\zeta, t): \zeta, t \in [0,T]\}$ have the representation
\begin{equation}\label{equation:infint}
\Koli_r(\zeta,t)\;=\;
x_r + t x_{r-1} + \frac{t^2}{2}  x_{r-2} + \ldots + \frac{t^r}{r!}  x_0 + T^r\sum_{k=1}^{\infty}\sqrt{\lambda_k}\,w_k(\zeta)\,f_{r,k}(t/T)
\end{equation}
with $f_{r,k}(t)=\int_0^t\dots\int_0^{s_1}f_k(s_0) \, {\d}{s_0} {\d}{s_1}\dots {\d}{s_{r-1}}$.

It follows from \eqref{eq:representation} and \eqref{equation:infint} that, for a fixed process time $T$, the random process \(\{\WW(\zeta): \zeta \in [0,T]\}\) obtained by
\begin{equation}\label{equation:algorep}
\WW(\zeta) \quad=\quad \LL^{-1}\DD(T^{-1})\left(\Kol(\zeta, T) - \DD(T) \HH \DD(T^{-1}) \xx\right)
\end{equation}
is a standard $(k+1)$-dimensional Brownian motion evolving in algorithmic time up to time \(T\).
Further, from the representation in \eqref{equation:infint}, it follows that the Brownian motion $\WW$, obtained in this way, does not depend on $T$. 
Thus we can write the components $\WWi_j$ of $\WW$ as linear combinations of the Brownian motions $w_k$:
\begin{equation}\label{equation:iterint}
\WWi_j(\zeta)\quad=\quad\sum_{k=1}^{\infty}w_k(\zeta)e_{j,k}\,,
\end{equation}
where the $e_{j,k}$ do not depend on $T$.

It will be convenient to define the infinite matrix $\EE$ whose $(j,k)$-th entry is given by $e_{j,k}$,
and to note that its rows must form an independent orthonormal set in the sequence space $l^2$.

We now describe the \emph{finite-look-ahead coupling} by constructing the coupled paths in each of the successive time intervals (look-ahead-blocks) \([T_1+\ldots+T_n,T_1+\ldots+T_{n+1}]\) using the Brownian motion $\mathcal{B}_n$ on the infinite-dimensional space $C[0,T_n]$ described in \eqref{equation:infBM}, for $T_n=\alpha^n$ for some fixed $\alpha>1$.
In the following, set $S_0=0$ and
$S_n=T_1+\ldots+T_n$
to be the time of initiation of the \(n^\text{th}\) look-ahead-block.

Before commencing detailed analysis, we provide a brief heuristic description of the coupling. 
Recall that the reflection coupling of Brownian motions started from distinct points gives a maximal coupling by using reflection on one Brownian path to produce the other (till the coupling time),
using the hyperplane that bisects the line joining their respective starting points. 
Although the Kolmogorov diffusion $\{\Kol(t): t \in [0,T]\}$ is not a Brownian motion in process time, 
the process $\{\WW(\zeta): \zeta \in [0,T]\}$ obtained from $\Kol(\zeta, T)$ in \eqref{equation:algorep} for a fixed process time $T$ evolves as a Brownian motion in algorithmic time. 
With this observation in mind, we couple two Kolmogorov diffusions $\Kol^{(1)}$ and $\Kol^{(2)}$ started from $\xx^{(1)}$ and $\xx^{(2)}$ respectively on the block $[0,T]$ as follows:
\begin{itemize}
\item[1.] Couple two infinite-dimensional Brownian motions $\{\mathcal{B}_n^{(1)}(\zeta,t): \zeta, T \in [0,T]\}$ 
and $\{\mathcal{B}_n^{(2)}(\zeta,t): \zeta, T \in [0,T]\}$ 
in such a way that the processes $\WW^{(1)}$ and $\WW^{(2)}$ obtained from $\mathcal{B}_n^{(1)}$ and $\mathcal{B}_n^{(2)}$ respectively 
via \eqref{equation:infint} and \eqref{equation:algorep} are reflection coupled (in algorithmic time) 
by reflecting in the hyperplane bisecting the initial discrepancy vector $\WW^{(1)}(0)-\WW^{(2)}(0)$.
\item[2.] Repeat this construction on each block $[S_n,S_{n+1}]$ updating the starting points of the respective Kolmogorov diffusions to $\Kol^{(1)}(S_n)$ and $\Kol^{(2)}(S_n)$.
\end{itemize}
This coupling of the infinite-dimensional Brownian motions projects down to a coupling of the corresponding driving Brownian motions (and hence the Kolmogorov diffusions) by setting
\begin{equation}\label{equation:projdown}
\Koli_0^{(i)}(t)\quad=\quad
\Koli_0^{(i)}(S_n)+\mathcal{B}_n^{(i)}(T_{n+1}, t-S_n)
\end{equation}
for $t \in [S_n,S_{n+1}]$ and $i=1,2$.

It is reasonable to expect that if such a coupling is achieved then it will be efficient,
as we are using reflection coupling of Brownian motions (in algorithmic time) in each block. 
The coupling is analysed in what follows, and efficiency of the coupling is shown in Theorem \ref{thm:finite-look-ahead}.

We will give an inductive description of the coupling for Kolmogorov diffusions $\Kol^{(1)}$ and $\Kol^{(2)}$ started from $\xx^{(1)}$ and $\xx^{(2)}$ respectively. Suppose we have constructed the coupling on $[0,S_n]$. If $\Kol^{(1)}(S_n)=\Kol^{(2)}(S_n)$, we can synchronously couple the Brownian motions $\Koli^{(1)}_0$ and $\Koli^{(2)}_0$ after $S_n$. So, henceforth we assume $\Kol^{(1)}(S_n)\neq\Kol^{(2)}(S_n)$. We will couple two infinite-dimensional Brownian motions $\mathcal{B}_n^{(1)}$ and $\mathcal{B}_n^{(2)}$ on the block $[0, T_{n+1}]$, which are represented by the Karhunen-Lo\`eve expansion
\begin{equation}\label{equation:infBMcouple}
\mathcal{B}_n^{(i)}(\zeta,t)\quad=\quad
\sum_{k=1}^{\infty}\sqrt{\lambda_k}\,w_{n,k}^{(i)}(\zeta)\,f_k(t/T_{n+1})
\end{equation}
for $\zeta, t \in [0,T_{n+1}]$ and $i=1,2$. 
The corresponding coupling for the Brownian motions $\Koli_0^{(1)}$ and $\Koli_0^{(2)}$ (and thus for the entire diffusions $\Kol^{(1)}$ and $\Kol^{(2)}$) on the block $[S_n,S_{n+1}]$ can then be obtained by \eqref{equation:projdown}.

The $(k+1)$ dimensional Brownian motions running in algorithmic time obtained from the iterated integrals by \eqref{equation:algorep} are denoted by $\WW_n^{(1)}$ and $\WW_n^{(2)}$ respectively.

Write $\ZZ_n \;=\; \Kol^{(1)}(S_n) - \Kol^{(2)}(S_n)$. 
Define the unit vectors $\displaystyle{\underline{\nu}_n=\frac{\LL^{-1} \DD(T_n^{-1}) \ZZ_n}{\|\LL^{-1} \DD(T_n^{-1}) \ZZ_n\|}}$ 
\Big(taking $\displaystyle{\underline{\nu}_0=\frac{\LL^{-1}\zz}{\|\LL^{-1}\zz\|}}$\Big)
and $\displaystyle{\underline{\eta}_n=\frac{(\LL^{-1} \HH\DD(\alpha^{-1})\LL) \underline{\nu}_n}{\|(\LL^{-1} \HH\DD(\alpha^{-1})\LL) \underline{\nu}_n\|}}$. 
Let $\{B_{n,j} : j \ge 1\}$ be independent standard scalar Brownian motions.
Recall the matrix $\EE$ defined just after \eqref{equation:iterint}. 
The infinite vector $\underline{v}_{n,1}=\EE^T\underline{\eta}_n$ has $l^2$-norm one.
We can extend it to an orthonormal basis of $l^2$, say $\{\underline{v}_{n,1}, \underline{v}_{n,2}\dots\}$. 
Let $\PP$ denote the (infinite) orthogonal matrix with columns formed by these vectors.
Note that \(\PP\) is a unitary matrix, and therefore the rows of $\PP$ are also orthonormal.
Now, we can define a coupling of the component Brownian motions $w^{(i)}_{n,k}$ in terms of the matrix $\PP$ and the Brownian motions $B_{n,j}$. 
Define the stopping time $\displaystyle{\sigma_n=\inf\left\lbrace \zeta>0: B_{n,1}(\zeta)=-\tfrac12{\|\LL^{-1} \HH\DD(T_{n+1}^{-1}) \ZZ_n\|}\right\rbrace}$. 
Then the coupling is a reflection coupling as follows:
\begin{eqnarray}\label{eqnarray:compcouple}
w^{(i)}_{n,k}(\zeta)\quad=\quad
\left\{
       \begin{array}{ll}
(-1)^{i+1}\PPi_{k,1}B_{n,1}(\zeta) + \sum_{j=2}^{\infty}\PPi_{k,j}B_{n,j}(\zeta) \ & \mbox{ when } \zeta \le \sigma_n\,,\\
w^{(i)}_{n,k}(\sigma_n)+\sum_{j=1}^{\infty}\PPi_{k,j}\left(B_{n,j}(\zeta)-B_{n,j}(\sigma_n)\right) \ & \mbox{ when } \zeta > \sigma_n\,,
       \end{array}
\right.       
\end{eqnarray}
for $i=1,2$. This gives the coupling between $\mathcal{B}_n^{(1)}$ and $\mathcal{B}_n^{(2)}$, and thus the corresponding Kolmogorov diffusions $\Kol^{(1)}$ and $\Kol^{(2)}$, via \eqref{equation:infBMcouple}.

Note that the reflection coupling recipe \eqref{eqnarray:compcouple} along with \eqref{equation:iterint} give us
\begin{equation}\label{equation:Wcouplealgo}
\WW_n^{(1)}(\zeta) -\WW_n^{(2)}(\zeta)\quad=\quad
2\, B_{n,1}(\zeta \wedge \sigma_n)\, \underline{\eta}_n\,.
\end{equation}
Thus the $(k+1)$ dimensional Brownian motions $\WW_n^{(1)}$ and $\WW_n^{(2)}$ (running in algorithmic time) 
are coupled in such a way that their difference is a rate \(4\) scalar Brownian motion running along the vector $\underline{\eta}_n$. 
This will be the crucial fact we will use to prove efficiency of this coupling strategy in the following theorem.

\begin{theorem}[Efficient finite-look-ahead coupling strategies]\label{thm:finite-look-ahead}
 The finite-look-ahead coupling of the Kolmogorov diffusion of index \(k\), constructed as above over successive intervals of lengths \(T_n=\alpha^n\) for some fixed \(\alpha>1\),
provides an efficient coupling strategy for the Kolmogorov diffusion.
\end{theorem}


\begin{remark}
 An efficient coupling \emph{strategy} for the Kolmogorov diffusion has to couple at a much faster rate 
 when the initial discrepancy vector \(\zz\) has an initial segment \(z_0, z_1, \ldots, z_{r-1}\) which is zero: see Theorem \ref{thm:max-coupling}.
\end{remark}

\begin{proof}
Consider the above coupling between two index \(k\) generalized Kolmogorov diffusions \(\Kol^{(1)}\), \(\Kol^{(2)}\), 
begun at \(\xx^{(1)}\), \(\xx^{(2)}\) respectively, with \(\zz=\xx^{(1)}-\xx^{(2)}\).
Setting \(\ZZ_0=\zz=\xx^{(1)}-\xx^{(2)}\) and $S_0=T_0=0$, and employing the representation \eqref{equation:algorep}, we may write for $n \ge 1$,
\begin{equation}\label{equation:zeqn}
\ZZ_n(\zeta) \;=\;
 \Kol^{(1)}(S_{n-1}+\zeta) - \Kol^{(2)}(S_{n-1}+\zeta) 
 \;=\; \DD(T_n)\Big(\HH \DD(T_n^{-1}) \ZZ_{n-1}(T_{n-1}) +  \LL \Delta_n \WW(\zeta)\Big)\,,
\end{equation}
where \(\Delta_n \WW(\zeta)=\WW_{n-1}^{(1)}(\zeta)-\WW_{n-1}^{(2)}(\zeta)\) and $\zeta \in [0,T_n]$. We will write $\ZZ_n$ for $\ZZ_n(T_n)$ and $\Delta_n\WW$ for $\Delta_n\WW(T_n)$.

Set $F_0(0)=\|\LL^{-1}\zz\|$ and $\displaystyle{\underline{\nu}_0=\frac{\LL^{-1}\zz}{\|\LL^{-1}\zz\|}}$. For $n \ge 1$, write $F_n(\zeta)=\|\LL^{-1} \DD(T_n^{-1}) \ZZ_n(\zeta)\|$ and recall the unit vectors $\displaystyle{\underline{\nu}_n =  \frac{\LL^{-1} \DD(T_n^{-1}) \ZZ_n}{\|\LL^{-1} \DD(T_n^{-1}) \ZZ_n\|}}$ and $\displaystyle{\underline{\eta}_n=\frac{(\LL^{-1} \HH\DD(\alpha^{-1})\LL) \underline{\nu}_n}{\|(\LL^{-1} \HH\DD(\alpha^{-1})\LL) \underline{\nu}_n\|}}$ (defined when $\Kol^{(1)}(S_n)\neq \Kol^{(2)}(S_n)$). As before, we will write $F_n$ for $F_n(T_n)$.

Then \eqref{equation:zeqn} gives us the following:
\begin{multline}\label{multiline:fgen}
 F_n(\zeta) \quad=\quad \|\LL^{-1} \DD(T_n^{-1}) \ZZ_n(\zeta)\| \quad=\quad
 \|\LL^{-1} \HH \DD(T_n^{-1})\ZZ_{n-1} + \Delta_n \WW(\zeta)\|\\
 \quad=\quad  \|F_{n-1} (\LL^{-1} \HH \DD(\alpha^{-1}) \LL) \underline{\nu}_{n-1} + \Delta_n \WW(\zeta)\|\,.
\end{multline}
Suppose $\Kol^{(1)}(S_{n-1})\neq \Kol^{(2)}(S_{n-1})$. Then from \eqref{equation:Wcouplealgo}, we have
\[
 \Delta_n \WW(\zeta) \quad=\quad 2\, B_{n-1,1}(\zeta \wedge \sigma_{n-1})\,  \underline{\eta}_{n-1}\,.
\]
Substituting this into \eqref{multiline:fgen}, 
and noting the definition of the stopping time \(\sigma_{n-1}\),
we get
\begin{equation}\label{equation:fgentwo}
F_n(\zeta)\quad=\quad
\left( F_{n-1}  + 2 \,\frac{ B_{n-1,1}(\zeta \wedge \sigma_{n-1}) }{\|(\LL^{-1} \HH \DD(\alpha^{-1}) \LL) \underline{\nu}_{n-1}\|}\right)\|(\LL^{-1} \HH \DD(\alpha^{-1}) \LL) \underline{\nu}_{n-1}\|\,.
\end{equation}
Suppose the coupling has not been successful in $[0,S_n]$. Then putting $\zeta=T_n$ in \eqref{multiline:fgen} and using \eqref{equation:Wcouplealgo} again, we proceed similarly as above to obtain the following recursive relation:
\begin{multline}\label{multiline:nudet}
 F_n \underline{\nu}_n 
 \quad=\quad \left( F_{n-1}  + 2\, \frac{ B_{n-1,1}(T_n \wedge \sigma_{n-1}) }{\|(\LL^{-1} \HH \DD(\alpha^{-1}) \LL) \underline{\nu}_{n-1}\|}\right)(\LL^{-1} \HH \DD(\alpha^{-1}) \LL) \underline{\nu}_{n-1}\,.
\end{multline}
Thus, for this sequential coupling scheme it follows from \eqref{multiline:nudet} that the vector \(\underline{\nu}_n\) is deterministic, and indeed
\begin{equation}\label{eqn:nu-evolution}
 \underline{\nu}_n \quad=\quad \frac{(\LL^{-1} \HH \DD(\alpha^{-1}) \LL) \underline{\nu}_{n-1}}{\|(\LL^{-1} \HH \DD(\alpha^{-1}) \LL) \underline{\nu}_{n-1}\|}
\;=\; \ldots \;=\;
 \frac{(\LL^{-1} (\HH \DD(\alpha^{-1}))^n \LL) \underline{\nu}_{0}}{\|(\LL^{-1} (\HH \DD(\alpha^{-1}))^n \LL) \underline{\nu}_{0}\|}\,.
\end{equation}
So we can re-express \eqref{equation:fgentwo} in terms of \(\underline{\nu}_0\):
\begin{equation}\label{equation:fgenthree}
F_n(\zeta)\;=\;
\left( F_{n-1}  + 2\, \frac{ B_{n-1,1}(\zeta \wedge \sigma_{n-1}) }{\|(\LL^{-1} \HH \DD(\alpha^{-1}) \LL) \underline{\nu}_{n-1}\|}\right) \frac{\|(\LL^{-1} (\HH \DD(\alpha^{-1}))^n \LL) \underline{\nu}_{0}\|}{\|(\LL^{-1} (\HH \DD(\alpha^{-1}))^{n-1} \LL) \underline{\nu}_{0}\|}\,.
\end{equation}
Consequently, if we define
$$
G_n(\zeta)\quad=\quad
\frac{F_n(\zeta)}{\|(\LL^{-1} (\HH \DD(\alpha^{-1}))^{n} \LL) \underline{\nu}_{0}\|}
$$
for $\zeta \in [0, T_n]$, then \eqref{equation:fgenthree} gives us
\begin{equation}\label{eq:interpolation}
 G_n(\zeta) \quad=\quad
  \left( 
 G_{n-1}
 + 2 \, \frac{ B_{n-1,1}(\zeta \wedge \sigma_{n-1})}{\|(\LL^{-1} (\HH \DD(\alpha^{-1}))^n \LL) \underline{\nu}_{0}\|}  \right)
  \,,
\end{equation}
for $\zeta \in [0,T_n]$, where as before, we write $G_{n-1}$ for $G_{n-1}(T_{n-1})$.

Consider now the (deterministic) evolution of the length \(\|(\LL^{-1} (\HH \DD(\alpha^{-1}))^n \LL) \underline{\nu}_{0}\|\).
The matrices \(\LL^{-1} \HH \DD(\alpha^{-1}) \LL\) and \(\HH \DD(\alpha^{-1})\) share the same eigenvalues.
Since \(\HH\) is lower-triangular, and has values \(1\) along the diagonal, and since \(\alpha^{-1}\in(0,1)\), it follows that the eigenvalues of \(\LL^{-1} \HH \DD(\alpha^{-1}) \LL\) are distinct and positive,
and are \(1\), \(\alpha^{-1}\), \ldots, \(\alpha^{-k}\). 
Consequently it follows that \(\LL^{-1} \HH \DD(\alpha^{-1}) \LL\) has \(k+1\) distinct eigenvectors; 
let \(\ee_0\), \(\ee_1\), \ldots, \(\ee_k\) be the eigenvectors corresponding to the eigenvalues \(1\), \(\alpha^{-1}\), \ldots, \(\alpha^{-k}\);
we suppose these to be of unit Euclidean norm. 
Note in particular that \(\ee_0=(1,0,\ldots,0)^\top\) while \(\ee_1\), \ldots, \(\ee_k\) all have zero initial coordinate (this follows from the lower-triangular nature of the matrices \(\LL\), \(\HH\), \(\DD(\alpha^{-1})\);
moreover the explicit construction of \(\LL\) implies that its top row is given by \((1,0,\ldots,0)\)).
Write \(\underline{\nu}_0=\gamma_0\ee_0+\gamma_1\ee_1+\ldots+\gamma_k\ee_k\).

We suppose first that \(\gamma_0\neq0\). From the above it follows that
\begin{equation}\label{equation:normbound}
 \|\LL^{-1} (\HH \DD(\alpha^{-1}))^{n} \LL \underline{\nu}_0\| \quad=\quad
 \|\gamma_0\ee_0 + \gamma_1\alpha^{-n}\ee_1+\ldots+\gamma_k\alpha^{-kn}\ee_{k}\|\,,
\end{equation}
which by the triangle inequality lies in the range
\begin{multline}\label{multiline:triangle}
|\gamma_0| - \max\{|\gamma_i|:i=1,\ldots,k\}\alpha^{-n}\\
 \quad\leq\quad \|\LL^{-1} (\HH \DD(\alpha^{-1}))^{n} \LL \nu_0\| \quad\leq\quad\\
 |\gamma_0| + \max\{|\gamma_i|:i=1,\ldots,k\}\alpha^{-n}\,.
\end{multline}
Choose \(n_0\) such that the lower bound of this range exceeds \(0\) for \(n\geq n_0\). 



Observe from the ``algorithmic time'' interpolation \eqref{eq:interpolation} that if we set $G(0)=F_0(0)=\|\LL^{-1}\zz\|$ and \(G(\zeta)=G_n(\zeta-S_{n-1})\) for $\zeta \in [S_{n-1},S_n]$, then $G$ is driven by
a time-changed Brownian motion built
up from the increments
\[
2 \, \frac{B_{n-1,1}((\zeta-S_{n-1}) \wedge \sigma_{n-1})}{\|\LL^{-1} (\HH \DD(\alpha^{-1}))^{n} \LL \underline{\nu}_0\|} \qquad \text{for }S_{n-1} < \zeta \leq S_{n}\,.
\]
Moreover, $G(\zeta)=0$ for some $\zeta \in [S_{n-1},S_n]$ (this happens when $\zeta-S_{n-1}=\sigma_{n-1}$) if and only if $F_n=0$ and thus the Kolmogorov diffusions $\Kol^{(1)}$ and $\Kol^{(2)}$ (running in ``process time") couple by time $S_n$. 
The ratio $S_n/S_{n-1}$ always equals $\alpha$,
so a simple asymptotic analysis shows that for coupling efficiency considerations it is sufficient to analyse the hitting time of zero for the process $G$.

We see from \eqref{multiline:triangle} that by time \(S_n=\alpha + \ldots + \alpha^n\) the intrinsic time of the time-changed Brownian motion will lie between the two values
\[
\sum_{m=1}^{n_0-1}\frac{4 \alpha^m}{\|\LL^{-1} (\HH \DD(\alpha^{-1}))^{m} \LL \underline{\nu}_0\|^2}
+
\sum_{m=n_0}^n \frac{4 \alpha^m}{(|\gamma_0|\mp  \max\{|\gamma_i|:i=1,\ldots,k\}\alpha^{-m})^2}
\]
and so the probability of \(G\) not yet having hit zero will be of order
\[
\frac{C_1}{\sqrt{S_n}}\quad\leq\quad
\left(
C_3 + \sum_{m=n_0}^n \frac{4 \alpha^m}{(|\gamma_0|\mp  \max\{|\gamma_i|:i=1,\ldots,k\}\alpha^{-m})^2}
\right)^{-1/2}
 \quad\leq\quad \frac{C_2}{\sqrt{S_n}}
\]
for positive constants \(C_1\), \(C_2\), \(C_3\).

Asymptotic analysis for large \(n\) now shows that efficiency follows in the case \(\gamma_0\neq0\).

Suppose \(\gamma_0=0\), so the first coordinate of \(\zz\) must vanish. In case the first \(r>0\) coordinates of \(\zz\) vanish, 
the deterministic evolution of \(\underline{\nu}_n\) given by \eqref{eqn:nu-evolution} together with the lower-triangular nature of \(\LL^{-1} \HH \DD(\alpha^{-1})\LL\) ensures that the  first \(r\) coordinates of \(\underline{\nu}_n\) also vanish,
for all \(n\).
Hence,
\begin{equation}\label{equation:normdec}
 \|(\LL^{-1} \HH \DD(\alpha^{-1}) \LL) \underline{\nu}_{n}\| \quad\leq\quad \alpha^{-r} \,.
\end{equation}
We now argue much as above, but working with the rather simpler comparison process
\[
 \widetilde{G}(t) \quad=\quad 2 \alpha^{nr} B_{n-1,1}((\zeta-S_{n-1}) \wedge \sigma_{n-1}) + \alpha^r\|(\LL^{-1} \HH \DD(\alpha^{-1}) \LL) \underline{\nu}_{n-1}\| \widetilde{G}(S_{n-1}) \,,
\]
for \(S_{n-1} < \zeta \leq S_{n}\), where \(\widetilde{G}(0)=\|\LL^{-1}\zz\|\).

From \eqref{equation:normdec}, $\alpha^r\|(\LL^{-1} \HH \DD(\alpha^{-1}) \LL) \underline{\nu}_{n-1}\| \le 1$. Thus, the jumps of the process $\widetilde{G}$ at the times $\{S_i\}_{i \ge 0}$ decrease the value of $\widetilde{G}$. Hence, the hitting time of zero for \(\tfrac12 \widetilde{G}(S_n)\) is dominated by that of a new Brownian motion run till time 
\[
\alpha^{2r} T_1 + \ldots + \alpha^{2nr} T_n \quad=\quad
\alpha^{2r+1} + \ldots + \alpha^{n (2r+1)} \sim {C_1}^{-2} \alpha^{(n+1)(2r+1)}\,.
\]
This avoids \(0\) up to this time with probability of order \(C_1 / \alpha^{(n+1)(2r+1)/2}\).
Since \(S_n \sim C_2 \alpha^{n+1}\), we deduce that the asymptotics for the probability of avoiding zero before \(S_n\) are of order at most \(1/S_n^{r+1/2}\).

As a consequence it follows that this finite-look-ahead coupling provides an efficient coupling strategy, taking efficient advantage of faster coupling when an initial set of coordinates of \(\zz\) vanish. 
This completes the proof.
\end{proof}

Suppose we require the range of precognition to be bounded throughout the coupling procedure:
is it still possible to produce efficient couplings, so long as the Brownian components of the two coupled Kolmogorov diffusions start at different points?
We do not yet have a general answer to this question,
but can show that the obvious approach cannot work.
 Consider the above finite-look-ahead coupling of the Kolmogorov diffusion of index \(k\), defined over successive intervals of length \(T_n=1\) (so in particular the look-ahead is \emph{bounded}).
Then this coupling may not even succeed!

 We justify this assertion by proceeding as in the proof of Theorem \ref{thm:finite-look-ahead}, but taking \(T_n=1\), so \(S_n=n\). 
 Then \eqref{equation:fgentwo} and \eqref{eqn:nu-evolution} simplify: we obtain
 \begin{align}
  F_n(\zeta) \quad&=\quad \|(\LL^{-1} \HH \LL) \underline{\nu}_{n-1}\| F_{n-1} + 2 B_{n-1,1}(\zeta \wedge \sigma_{n-1}) \,,\nonumber\\
  \underline{\nu}_n \quad&=\quad \frac{(\LL^{-1} \HH  \LL) \underline{\nu}_{n-1}}{\|(\LL^{-1} \HH \LL) \underline{\nu}_{n-1}\|}
  \quad=\quad
  \frac{\LL^{-1} \HH(n) \zz}{\|\LL^{-1} \HH(n) \zz\|}
  \,.
  \label{eqn:fixed-nu-evolution}
\end{align} 
Accordingly we obtain
\[
F_n(\zeta) \quad=\quad \frac{\|\LL^{-1} \HH(n) \zz\|}{\|\LL^{-1} \HH(n-1) \zz\|} F_{n-1} + 2 B_{n-1,1}(\zeta \wedge \sigma_{n-1}) \,.
\]
But asymptotically (considering the action of \(H(n)\) on \(\zz\) for large \(n\))
\[
 \frac{\|\LL^{-1} \HH(n) \zz\|}{\|\LL^{-1} \HH(n-1) \zz\|} \quad\sim\quad \left(\frac{n}{n-1}\right)^k
\]
(note that the index \(k\) satisfies \(k\geq1\)). 
Therefore the time-changed Brownian interpolation of the \(F_n/n^k\) 
(following the proof of Theorem \ref{thm:finite-look-ahead}) may be compared with a scalar Brownian motion run up to time
\[
  \sum_n \frac{1}{n^{2k}}\,.
\]
This sum converges, and therefore there is a positive probability of the Brownian motion not reaching zero.

\section{Conclusion} \label{sec:conc}
In this paper we have studied rates at which the (generalized) Kolmogorov diffusion can be coupled using Markovian or near-Markovian couplings. 
While this diffusion does possess successful Markovian couplings \cite{BenArousCranstonKendall-1995,KendallPrice-2004}, such couplings cannot be maximal (Theorem \ref{thm:noMarkKol}),
nor can there be an efficient Markovian coupling strategy (Theorem \ref{thm:notefficient}).
Moreover, at least in the classical case, there can be no optimal Markovian coupling (Theorem \ref{thm:noopt}).
By way of compensation, it \emph{is} possible to exhibit a simple efficient finite-look-ahead coupling strategy even for the generalized Kolmogorov diffusion (Theorem \ref{thm:finite-look-ahead}). 
Thus a controlled amount of anticipation suffices to obtain efficiency of coupling.

The results of \cite{BanerjeeKendall-2014} show that smooth elliptic diffusions only admit Markovian maximal couplings in cases of very special geometry.
The Kolmogorov diffusion is the simplest non-trivial nilpotent diffusion,
and so its failure to admit Markovian maximal couplings suggests the conjecture that no smooth non-trivial nilpotent diffusions
can admit Markovian maximal couplings. Certainly this seems plausible for the case of planar Brownian motion plus It\^o stochastic area 
\cite{BenArousCranstonKendall-1995,Kendall-2007,Kendall-2009d}: 
perhaps the conjecture can be resolved in the manner of \cite{BanerjeeKendall-2014}, using ideas from the geometry of nilpotent groups. 

The above results can be compared with those of \cite{Kendall-2013a}, concerning the coupling of scalar Brownian motion together with its local time at zero.
That case lies outside the range of nilpotent diffusions, however it does seem to provide relevant insights.
For the local time coupling a (partly numerical) argument shows that there is no Markovian maximal coupling, 
but a control-theoretic argument shows that a simple reflection / synchronous coupling is optimal Markovian.
It would be most interesting to determine whether the existence of efficient Markovian couplings for a given smooth diffusion, or of optimal Markovian couplings, 
can enforce geometric rigidity in a manner similar to the existence of maximal Markovian couplings.
(The results in \cite[Sections 3, 4]{BurdzyKendall-2000} can be viewed as providing a very preliminary exploration to this kind of problem in the special case 
of reflecting Brownian motion in a compact convex domain.)
This would further elucidate the enigmatic r\^ole of geometry in probabilistic coupling theory.

A further detailed question for future research is, whether it is possible to attain efficiency for a \emph{bounded horizon} finite-look-ahead coupling
for a Kolmogorov diffusion.
The coupling described in Theorem \ref{thm:finite-look-ahead} has a horizon which extends at a geometric rate over successive blocks: 
as noted after the proof of Theorem \ref{thm:finite-look-ahead}, the obvious bounded horizon coupling does not even have the property of being successful.
There are instances in which a small amount of look-ahead can have a dramatic influence on coupling rate -- see the work of Smith \cite{Smith-2011a} on coupling for a Gibbs' sampler on the simplex --
it would be very interesting if one could map out the circumstances in which this might apply in the relatively well-behaved case of smooth diffusions.
As exemplified in \cite{Smith-2011a}, probabilistic coupling has a large part to play in the analysis of random algorithms;
it would be of considerable advantage to gain some case history for the potential of modestly non-Markovian coupling 
to deliver faster couplings in the amenable instance of smooth diffusions.

\ack
This work was supported by EPSRC Research Grant EP/K013939.

We are grateful to an anonymous referee whose suggestions greatly improved the paper.

\bibliographystyle{apt}
\bibliography{Kolmogorov}

\newpage           

\end{document}